\documentclass[12pt,reqno]{amsart}
\usepackage[margin=1in]{geometry}
\usepackage{amscd,amsfonts,amsmath,amssymb,amsthm,latexsym,mathrsfs,textcomp,verbatim}
\usepackage{accents}
\usepackage{bm}
\usepackage{bbm}
\usepackage{booktabs}
\usepackage{cite}
\usepackage{color}
\usepackage{constants}
\usepackage{csquotes}
\usepackage{enumerate}
\usepackage{etoolbox}
\usepackage{float}
\usepackage{hypbmsec}
\usepackage[dvipsnames]{xcolor}
\usepackage[breaklinks=true,colorlinks=true,linkcolor=black!25!RoyalBlue,citecolor=black!25!RoyalBlue,filecolor=RoyalBlue,urlcolor=teal]{hyperref}
\hypersetup{linktocpage}
\usepackage[capitalise]{cleveref}
\usepackage{mathtools}
\usepackage{soul}
\usepackage{tikz}
\usepackage{tikz-cd}
\usetikzlibrary{shapes.geometric}
\usetikzlibrary{shapes.misc}
\usetikzlibrary{positioning}
\allowdisplaybreaks[4]
\usepackage{hypbmsec}
\newtheorem{theorem}{Theorem}[section]
\newtheorem{corollary}[theorem]{Corollary}
\newtheorem{proposition}[theorem]{Proposition}
\newtheorem{lemma}[theorem]{Lemma}
\newtheorem{lettertheorem}{Theorem}

\newtheorem{lettercorollary}{Corollary}

\newtheorem{conjecture}[theorem]{Conjecture}

\newtheorem{remark}[theorem]{Remark}
\newtheorem{definition}[theorem]{Definition}
\newtheorem{question*}{Question}
\newtheorem{problem*}{Problem}

\theoremstyle{definition}

\newtheorem{example}[theorem]{Example}

\numberwithin{equation}{section}

\newcommand{\Z}{\mathbb{Z}}

\newcommand{\R}{\mathbb{R}}

\newcommand{\GL}{\mathrm{GL}}

\newcommand{\rk}{\mathrm{rk}}

\newcommand{\Li}{\mathrm{Li}}
\renewcommand{\Im}{\mathrm{Im}}

\renewcommand{\pmod}[1]{\, (\mathrm{mod} {\, #1})}

\newcommand{\ord}{\mathop{\mathrm{ord}}}

\renewcommand{\Re}{\mathrm{Re}}

\def\res{\mathop{\mathrm{res}}}

\newcommand{\Tr}{\textbf{Tr}}

\patchcmd{\section}{\scshape}{\bfseries}{}{}
\makeatletter
\renewcommand{\@secnumfont}{\bfseries}
\makeatother

\makeatletter\newcommand{\tpmod}[1]{{\@displayfalse\pmod{#1}}}
\makeatletter
\@namedef{subjclassname@2020}{\textup{2020} Mathematics Subject Classification}
\makeatother

\begin{document}

\title{Euler Products at the centre and applications to Chebyshev's Bias}

\author{Arshay Sheth}

\address{Mathematics Institute, Zeeman Building, University of Warwick, Coventry, CV4 7AL, UK}

\thanks{The author is supported by funding from the European Research Council under the European Union’s Horizon 2020 research and innovation programme (Grant agreement No. 101001051 — Shimura varieties and the Birch--Swinnerton-Dyer conjecture).}
\email{arshay.sheth@warwick.ac.uk}
\urladdr{\href{https://arshaysheth.github.io}{https://arshaysheth.github.io}}


\begin{abstract}
Let $\pi$ be an irreducible cuspidal automorphic representation of $\GL_n(\mathbb A_\mathbb Q)$ with associated $L$-function $L(s, \pi)$.  %We prove, under the assumption of the Riemann Hypothesis and the Ramanujan Conjecture for $L(s, \pi)$, that if we let $m=\ord_{s=\frac{1}{2} }L(s, \pi)$, then 
%$$
%(\log x)^m \cdot  \prod_{p \leq x}  \prod_{j=1}^n (1-\alpha_{j, p}p^{-\frac{1}{2} })^{-1}   \asymp 1 , 
%$$
%for all $x$ outside a set of finite logarithmic measure. Conversely, we also show, again assuming the Ramanujan Conjecture, that the above order of magnitude statement implies the Riemann Hypothesis for $L(s, \pi)$.  
We study the behaviour of the partial Euler product of $L(s, \pi)$ at the center of the critical strip. Under the assumption of the Generalized Riemann Hypothesis for $L(s, \pi)$ and assuming the Ramanujan--Petersson conjecture when necessary, we  establish an asymptotic, off a set of finite logarithmic measure, for the partial Euler product at the central point  that confirms a conjecture of Kurokawa. As an application, we obtain results towards Chebyshev's bias in the recently proposed framework of Aoki-Koyama \cite{AokiKoyama2023} . 
\end{abstract}

\maketitle
%\tableofcontents

\section{Introduction}

In this paper, we study the behaviour of partial Euler products of automorphic $L$-functions at the central point of their critical strip.  The motivation for investigating the behaviour of Euler products at the central point dates back
to the original version of the Birch and Swinnerton-Dyer conjecture.

\begin{conjecture}[Birch and Swinnerton-Dyer ~\cite{BirchSwinnertonDyer1965}] 
\label{OBSD}
Let $E/\mathbb Q$ be an elliptic curve with rank $r$ and for each prime $p$, let $N_p=\#E_{\mathrm{ns}}(\mathbb F_p)$, where $E_{\mathrm{ns}}(\mathbb F_p)$ denotes the set of non-singular $\mathbb F_p$-rational points on a minimal Weierstrass model for $E$ at $p$. Then we have that $$\prod_{p \leq x} \frac{N_p}{p}  \sim C (\log x)^{r}$$ as $x \to \infty$ for some constant $C$ depending on $E$. 
\end{conjecture}

This conjecture can be viewed as an assertion about the asymptotics of the partial Euler product at the central point $s=1$ of the $L$-function $L(E, s)$ attached to $E$.  If we let $N_E$ denote the conductor of $E$ and define $a_p=p+1-N_p$ for $p \nmid N_E$ and let $a_p=p-N_p$ for $p | N_E$,  $L(E, s)$ is defined for $\Re(s) > \frac{3}{2}$  by 
\begin{equation*}\label{ellipticcurveL}
L(E, s) \coloneqq \prod_{p \mid N_E} \frac{1}{1-a_{p} p^{-s}} \prod_{p \nmid N_E} \frac{1}{1-a_{p}p^{-s}+p^{1-2s}}. 
\end{equation*}
By the work of Wiles ~\cite{Wiles1995} and Breuil--Conrad--Diamond--Taylor ~\cite{BreuilConradDiamondTaylor2001},  $L(E, s)$ admits an analytic continuation to the the entire complex plane and has a functional equation relating values at $s$ and $2-s$. Defining 
$$
P_E(x)= \prod_ {\substack{p \leq x \\ p | N_E}} \frac{1}{1-a_{p} p^{-1}} \prod_{\substack{p \leq x \\ p \nmid N_E}} \frac{1}{1-a_{p}p^{-1}+p^{-1}}
$$
to be the partial Euler product at $s=1$,  Conjecture \ref{OBSD} can be reformulated to assert that 
$$
P_E(x) \sim \frac{1}{ C (\log x)^r }
$$
as $x \to \infty$.  Conjecture \ref{OBSD} has since motivated the study of partial Euler products of $L$-functions in their critical strip. For instance, if $\chi$ is a non-trivial Dirichlet character with associated Dirichlet $L$-function $L(s, \chi)$, Conrad has shown (see ~\cite[Theorem 3.3]{Conrad2005}) that the equality 
\begin{equation} \label{limpartial}
\lim\limits_{x \to \infty} \prod \limits _{p \leq x} (1-\chi(p)p^{-s})^{-1}=L(s, \chi)
\end{equation}
for all $s$ with $\textrm{Re}(s)>\frac{1}{2}$ is equivalent to the Generalized Riemann Hypothesis for $L(s, \chi)$.  More generally, in the case of entire $L$-functions, Conrad showed that the convergence of the Euler product in the right-half of the critical strip is equivalent to the Generalized Riemann Hypothesis for the $L$-function. 

It is natural to investigate the convergence of the Euler product on the critical line as well.  It is believed that, except for zeros on the critical line, Euler products of entire $L$-functions should also converge everywhere on the critical line, and the limit of the Euler product at a point on the line should equal the value of the $L$-function at the point \textit{i.e.} the analog of Equation \eqref{limpartial} should also hold everywhere on the critical line except on the set containing the zeros of the $L$-function. However, at the central point, there is often an unexpected factor of $\sqrt 2$ that is known to appear in the Euler product asymptotics. 

For instance, Goldfeld \cite{Goldfeld1982} showed that if Conjecture \ref{OBSD} is true,  then $C=\frac{r!} {L^{(r)}(E, 1)} \cdot \sqrt 2 e^{r \gamma}$, where $\gamma$ is Euler's constant; thus if $P_E(x)$, the partial Euler product at the centre, converges to a non-zero value as $x \to \infty$, then its value is $L(E, 1)/\sqrt 2$ (as opposed to simply $L(E, 1)$).  A conceptual explanation of the unexpected appearance of $\sqrt 2$ was subsquently given by Conrad in terms of second-moment $L$-functions.  If an $L$-function $L(s)$, which we henceforth assume is normalized so its centre is at $s=\frac{1}{2}$, is given by an Euler product 
$$
L(s)=\prod_p \prod_{j=1}^n (1-\alpha_{j, p}p^{ -s} )^{-1}, 
$$
its second moment $L$-function is given by 
$$
L_2(s) = \prod_p \prod_{j=1}^n (1-\alpha_{j, p}^2p^{ -s} )^{-1}
$$
and in practice is the ratio of the corresponding symmetric square $L$-function and the exterior square $L$-function.  Let $R=\ord \limits _{s=1}L_2(s)$; Conrad showed that if the Euler product at the center converges, then its value equals $ L(\frac{1}{2})/\sqrt 2^{R}$.  

\begin{example} \label{Dirichlet}
If $\chi$ is a Dirichlet character, $L_2(\chi, s)=L(\chi^2, s)$. Hence, if $\chi$ is a quadratic character, then $R=-1$; thus for a quadratic character,  if $\lim\limits_{x \to \infty} \prod \limits _{p \leq x} (1-\chi(p)p^{-1/2})^{-1}$ exists, then 

$$
\lim_{x \to \infty} \prod_{p \leq x} (1-\chi(p)p^{-1/2})^{-1}= \sqrt 2 \cdot L\left (\frac{1}{2}, \chi \right). 
$$
\end{example}

Based on the above phenomena Kurokawa \textit{et al.} (see for instance \cite{KanekoKoyamaKurokawa2022} or \cite{KimuraKoyamaKurokawa2014})  formulated a general conjecture about the convergence of partial Euler products at the centre of the critical strip.  This conjecture has been called the ``Deep Riemann Hypothesis",  since it not only implies the Generalized Riemann Hypothesis but, as we explain below, seems in a precise sense to lie deeper than the Generalized Riemann Hypothesis.

We now briefly explain this conjecture in the setting of general automorphic $L$-functions attached to an irreducible cuspidal automorphic representation $\pi$ of $\GL_n(\mathbb A_\mathbb Q)$. We choose to work in this very general setting since in the applications in Section \ref{Section4} of the paper, we consider a wide range of $L$-functions (attached to Dirichlet characters, elliptic curves and modular forms), and it is useful for us to have a general framework that encompasses these different $L$-functions. Indeed, according to the Langlands conjectures, the most general $L$-functions can all be expressed as products of $L$-functions attached to cuspidal automorphic representations of $\GL_n(\mathbb A_\mathbb Q)$, and all the $L$-functions needed in our applications are known to be examples of automorphic $L$-functions. Any automorphic $L$-function (see Section \ref{Section2}) can be written in the form
$$
L(s, \pi)=\prod_p \prod_{j=1}^n (1-\alpha_{j, p}p^{-s})^{-1},  
$$
where the $\alpha_{j, p}$'s are the Satake parameters for the local representation $\pi_p$.  We let  $\nu(\pi) = m(\mathrm{sym}^{2} \pi)-m(\wedge^{2} \pi) \in \Z$,  where $m(\rho)$ denotes the multiplicity of the trivial representation $\textbf{1}$ in $\rho$.  If we let $L_2(s, \pi)$ denote the second-moment $L$-function associated to $L(s, \pi)$,  we have that $\nu(\pi)=-R(\pi)$, where as above $R(\pi):=\ord_{s=1} L_2(s, \pi)$. Throughout this paper, we also assume that $L(s, \pi)$ is entire.

\begin{conjecture}[Kaneko-Koyama-Kurokawa \cite{KanekoKoyamaKurokawa2022}]\label{DRH}
Keep the assumptions and notation as above. Let $m = \ord_{s = 1/2} L(s, \pi)$. Then the limit
\begin{equation}\label{limit}
\lim_{x \to \infty} \left((\log x)^{m} \prod_{p \leq x} \prod_{j=1}^n \left(1-\alpha_{j, p} p^{-\frac{1}{2}} \right)^{-1} \right)
\end{equation}
satisfies the following conditions:
\begin{description}
\item[(A)] The limit~\eqref{limit} exists and is nonzero.
\item[(B)] The limit~\eqref{limit} satisfies
\begin{equation*}
\lim_{x \to \infty} \left((\log x)^{m} \prod_{p \leq x} \prod_{j=1}^n \left(1-\alpha_{j, p} p^{-\frac{1}{2}} \right)^{-1} \right)
 = \frac{\sqrt{2}^{ \nu (\pi) }}{e^{m \gamma} m!} \cdot L^{(m)} \left(\frac{1}{2}, \pi \right).
\end{equation*}
\end{description}
\end{conjecture}

By truncating the Euler product at large $x$ for a wide range of $L$-functions, there is now numerical evidence for Conjecture \ref{DRH}  (see for instance \cite[page 281]{Conrad2005} or \cite[Section 3] {KimuraKoyamaKurokawa2014}). 
As another piece of evidence, we also mention that the function field analog of the conjecture has been proven by Kaneko--Koyama--Kurokawa (see \cite[Theorem 5.2]{KanekoKoyamaKurokawa2022}). 
To explain the relation between Conjecture \ref{DRH} and the Generalized Riemann Hypothesis (for the $L$-function $L(s, \pi)$),  let $
a_{\pi}(p^k)=\alpha_{1, p}^k+ \cdots +\alpha_{n, p}^k 
$ and let $$\psi(x, \pi)= \sum_{p^k \leq x} \log p \cdot  a_\pi(p^k).$$ 
Then the Generalized Riemann Hypothesis is equivalent to the fact that 
\begin{equation} \label{GRHEstimate}
\psi(x, \pi)= O( \sqrt x (\log x)^2)
\end{equation}
while Conjecture \ref{DRH} is equivalent to the estimate 
\begin{equation} \label{DRHestimate}
\psi(x, \pi)= o( \sqrt x \log x). 
\end{equation}
It is in this sense that Conjecture $\ref{DRH}$ seems deeper than the Generalized Riemann Hypothesis.  The estimate in Equation \eqref{DRHestimate} is indeed plausible; an analysis of Montgomery \cite{Mon70} about the vertical distribution of zeros of $L$-functions on the critical line suggests that the true order of magnitude of $\psi(x, \pi)$ is at most  $O( \sqrt x ( \log \log \log x)^2)$ which would imply \eqref{DRHestimate}.

However, error terms may not be the best way to determine to the precise relation between the Generalized Riemann Hypothesis and Conjecture \ref{DRH}; for instance, the Generalized Riemann Hypothesis is also equivalent to the slightly weaker error term $\psi(x, \pi)= O(x^{\frac{1}{2}+\epsilon})$ for any $\epsilon>0$.  Since Conjecture \ref{DRH} concerns the convergence of the Euler product at the centre of the critical strip, it is natural to ask whether the Generalized Riemann Hypothesis can be related to the Euler product at the centre as well.  Our first result answers this question in the affirmative, by showing that in fact the Generalized Riemann Hypothesis implies Conjecture \ref{DRH} outside a set of finite logarithmic measure. We first recall the following definition. 

\begin{definition}
Let $S \subseteq \mathbb R_{\geq 2}$ be a measurable subset of the real numbers. The logarithmic measure of $S$ is defined to be 
$$
\mu^{\times}(S)= \int_{S} \frac{dt}{t}. 
$$
\end{definition}

We remark that since we are working in this level of generality, we also need to assume the Ramanujan--Petersson conjecture for $\pi$ (we refer to Section \ref{Section2} for the description of this conjecture and for the cases it has been proven).

\begin{lettertheorem}[Theorem \ref{boundedness}] \label{Thm1}
Let $\pi$ be an irreducible cuspidal automorphic representation of $\GL_n(\mathbb A_\mathbb Q)$ such that $L(s, \pi)$ is entire and let $m = \ord \limits _{s = \frac{1}{2}} L(s, \pi)$. Assume the Ramanujan--Petersson Conjecture and the Riemann Hypothesis for $L(s, \pi)$. Then there exists a subset $S \subseteq \mathbb R_{\geq 2}$ of finite logarithmic measure such that for all $x \not \in S$, 
$$
(\log x)^m \cdot  \prod_{p \leq x}  \prod_{j=1}^n (1-\alpha_{j, p}p^{-\frac{1}{2} })^{-1}   \sim  \frac{\sqrt{2}^{ \nu (\pi) }}{e^{m \gamma} m!} \cdot L^{(m)} \left(\frac{1}{2}, \pi \right).
$$

\end{lettertheorem}

The method of proof first consists of developing a suitable version of an explicit formula for $L(s, \pi)$ (Proposition \ref{explicit}) that allows us to establish the asymptotic behaviour of partial Euler products of $L(s, \pi)$ in the right-half of the critical strip (Theorem \ref{mainthm}) .  One of the terms in this asymptotic formula involves contributions coming from the zeros of $L(s, \pi)$; this term is the most delicate to handle, and the standard bound in Equation \eqref{GRHEstimate} coming from the Generalized Riemann Hypothesis will not suffice.  Rather, we need the  refined estimate $
\psi(x, \pi)= O( \sqrt x (\log \log x)^2)
$ which holds, conditional on the Generalized Riemann Hypothesis, outside a set of finite logarithmic measure (Theorem \ref{refinement}). 
The method described here builds upon previous work of the author \cite{Sheth2023}, where these techniques were used to study the relations between the original and modern formulations of the Birch and Swinnerton--Dyer conjecture. In this paper, we further explore applications of these ideas to questions concerning Chebyshev's bias.

\begin{comment}
Using the method of Goldfeld \cite{Goldfeld1982}, we are also able to prove the converse.  

\begin{theorem}
Let $\pi$ be an irreducible cuspidal automorphic representation of $\GL_n(\mathbb A_\mathbb Q)$ with  $m = \ord_{s = 1/2} L(s, \pi)$. 
Assume the Ramanujan conjecture for $L(s, \pi)$. 
Suppose that 
$$
(\log x)^m \prod_{p \leq x}  \prod_{j=1}^n (1-\alpha_{j, p}p^{-\frac{1}{2} })^{-1}   \asymp 1.  
$$
Then the Riemann Hypothesis for $L(s, \pi)$ holds. 

\end{theorem}

Thus, our results give a conceptual explanation of the relationship between Conjecture \ref{DRH} and the Riemann Hypothesis: while Conjecture \ref{DRH} predicts that Euler products should converge at the center, the slightly weaker Generalized Riemann Hypothesis is essentially equivalent to the fact that Euler products are bounded at the centre. 

\begin{comment}
We summarize this relationship in the following table. 
\vspace{5mm}
\\
\begin{center}
\begin{tabularx}{0.8\textwidth} { 
  | >{\raggedright\arraybackslash}X 
  | >{\centering\arraybackslash}X 
  | >{\raggedleft\arraybackslash}X | }
 \hline
 Conjecture \ref{DRH} (``Deep Riemann Hypothesis")  & Convergence of Euler products at the centre  \\
 \hline
 Generalized Riemann Hypothesis  & Boundedness of Euler products at the centre  \\
\hline
\end{tabularx}
\end{center}
\vspace{2mm}
\\
\end{comment}

\subsection{Applications to Chebyshev's bias}
Chebyshev's bias originally referred to the phenomenon that, even though the primes are equidistributed in the multiplicative residue classes mod 4, there seem to more primes congruent to 3 mod 4 than 1 mod 4.  Let $\pi(x; q, a)$ denote the number of primes up to $x$ which are congruent to $a$ modulo $q$ and let $S=\{x \in \mathbb R_{\geq 2}: \pi(x; 4, 3)-\pi(x; 4, 1)>0\}$.  Building on Chebyshev's observations, Knapowski--Tur\'{a}n~\cite{KnapowskiTuran1962} conjectured that the proportion of postive real numbers which lie in the set $S$ would equal 1 as $x \to \infty$.  However, this conjecture was later disproven by Kaczorowski~\cite{Kaczorowski1995}  conditionally on the Generalized Riemann Hypothesis, by showing that the limit does not exist. Rubinstein and Sarnak \cite{RubinsteinSarnak1994} instead considered the logarithmic density $\delta(S):= \frac{1}{\log X} \cdot \lim \limits _{X \to \infty}  \int_{t \in S \cap [2, X]} \frac{dt}{t}$ of $S$; assuming the Generalized Riemann Hypothesis and the assumption that the non-negative imaginary parts of zeros of Dirichlet $L$-functions are linearly independent over $\mathbb Q$, they showed that this limit exists and $\delta(S)=0.9959 \ldots$, hence giving a satisfactory explanation of this phenomenon. Similar biases have since been observed in various other situations as well. For instance, if $E/\mathbb Q$ is an elliptic curve and $a_p$ denotes the trace of the Frobenius at the prime $p$, then even though $a_p$ is positive and negative equally often by the Sato--Tate conjecture, Mazur \cite{Maz08} noted 
$
D_E(x) = \# \{ p \leq x: a_p>0\}-  \# \{ p \leq x: a_p <0\}
$
has a bias towards being negative if the rank of $E$ is large.  An explanation of this fact was subsquently given by Sarnak \cite{Sar07} in the spirit of \cite{RubinsteinSarnak1994}.  A conceptual framework for dealing with problems related to Chebyshev's bias, generalizing the Rubinstein--Sarnak approach to a wide range of $L$-functions, was recently given by Devin \cite{Devin2020}.

In \cite{AokiKoyama2023},  Aoki-Koyama present an alternative approach to describe Chebyshev's bias that is closely related to the behaviour of Euler products at the center of the critical strip. 

\begin{definition}[Aoki-Koyama \cite{AokiKoyama2023}] \label{Chebyshev}
Let $(c_p)_p \subseteq \R$ be a sequence over primes $p$  such that
$$
\lim_{x \to \infty} \frac{\#\{p \mid c_p > 0,  p \leq x \}}{\#\{p \mid c_p < 0,  p \leq x \}} = 1.
$$
We say that $(c_p)_p$ has a \textit{Chebyshev bias towards being positive} if there exists a positive constant $C$ such that 
$\displaystyle{
\sum_{p \leq x} \frac{c_p}{\sqrt p} \sim C \log \log x}$.  On the other hand, we say that $c_p$ is \textit{unbiased} if
$
\displaystyle{\sum_{p \leq x} \frac{c_p}{\sqrt p } = O(1)}. 
$
\end{definition}
For instance, if we let $\chi_4$ to be the non-trivial Dirichlet character modulo $4$, $c_p=\chi_4(p)$, then the sum in Definition \ref{Chebyshev} becomes 
$
\sum_{p \leq x} \frac{ \chi(p) } { \sqrt p} = \pi_{\frac{1}{2}}(x; 4, 1)-\pi_{\frac{1}{2}}(x; 4, 3), 
$
where 
\newline 
$\displaystyle{\pi_{s}(x; q, a) = \sum \limits _{\substack{p < x \colon \text{prime} \\ p \equiv a \tpmod q}}  \frac{1}{p^{s} } }$ for $s \geq 0$ is the weighted prime counting function. 

%In general, partial Euler products of $L$-functions at the central point are closely related (after taking logarithms) to the sum appearing in Definition \ref{Chebyshev}. 
Using Theorem \ref{Thm1}, we obtain the following asymptotic for the types of sums appearing in Definition \ref{Chebyshev}. 

\begin{lettertheorem}[Theorem \ref{Satake}] \label{satakebias}
Let $\pi$ be an irreducible cuspidal automorphic representation of $\GL_n(\mathbb A_\mathbb Q)$ such that $L(s, \pi)$ is entire and let $m= \ord_{s=1/2} L(s, \pi)$. Assume the Ramanujan--Petersson Conjecture and the Riemann Hypothesis for $L(s, \pi)$. Then there exists a constant $c_\pi$ such that 
$$
\sum_{p \leq x} \frac{\alpha_{1, p}+ \cdots +\alpha_{n, p}}{\sqrt p}= \left( \frac{R(\pi)}{2}-m \right) \log \log x+ c_\pi+ o(1)
$$
for all $x$ outside a set of finite logarithmic measure, where $R(\pi) =\ord \limits _{s=1} L_2(s, \pi)$. 
\end{lettertheorem}

As applications of this theorem, we obtain results towards Chebyshev's bias, in the sense of Definition \ref{Chebyshev}, for a wide class of equidistributed sequences in number theory. 

\begin{lettercorollary}[Corollary \ref{primenumberrace}] \label{corA}
Let $\chi_4$ denote the non-trivial Dirichlet character modulo 4.  Assume the Riemann Hypothesis for $L(\chi_4, s)$. 
Then there exists a constant $c$ such that 
$$
\sum_{p \leq x} \frac{\chi_4(p)}{\sqrt p}= - \frac{1}{2} \log \log x+ c+ o(1)
$$
for all $x$ outside a set of finite logarithmic measure.  In particular, in the sense of Definition \ref{Chebyshev}, there is a Chebyshev bias towards primes which are 3 mod 4. 
\end{lettercorollary}

\begin{lettercorollary} [Corollary \ref{tauasymptotic}] \label{corB}
Let $\tau(n)$ denote Ramanujan's tau function. Assume the Riemann Hypothesis for $L(s, \Delta)$. Then there exists a constant $c$ such that 
$$
\sum_{p \leq x} \frac{\tau(p)}{p^6} = \frac{1}{2} \log \log x + c+o(1)
$$
for all $x$ outside a set of finite logarithmic measure. In particular, in the sense of Definition \ref{Chebyshev}, the sequence $\tau(p) p^{-\frac{11}{2}}$ has a Chebyshev bias towards being positive.
\end{lettercorollary}

\begin{lettercorollary} [Corollary \ref{frobtrace}] \label{corC}
Let $E/\mathbb Q$ be an elliptic curve with rank $\rk(E)$ and Frobenius trace $a_p$ for each prime $p$. Assume the Birch and Swinnerton-Dyer conjecture and the Riemann Hypothesis for $L(E, s)$. 
Then there exists a constant $c$ depending on $E$ such that 
$$
\sum_{p \leq x} \frac{a_p}{p}= \left(\frac{1}{2}-\rk(E) \right) \log \log x+ c+o(1)
$$
for all $x$ outside a set of finite logarithmic measure.  In particular, $a_p$ has a bias towards being positive if $\rk(E)=0$ and a bias towards being negative if $\rk(E)>0$. 
\end{lettercorollary}

Corollary \ref{corC} is a concrete instance of Birch and Swinnerton-Dyer's initial observations, which motivated them to formulate their celebrated conjecture, that elliptic curves which have more rational points also tend to have more than the expected number of points modulo primes $p$. 

Corollaries \ref{corA}, \ref{corB} and \ref{corC} have been inspired by similar statements in \cite{AokiKoyama2023}, \cite{KK22} and \cite{KanekoKoyama2023} respectively. The key difference is that the statements in \textit{op.cit.} assume Conjecture \ref{DRH}, but our statements only assume the Generalized Riemann Hypothesis, with the caveat that our asymptotics hold for all $x$ outside an exceptional set. Indeed, as far as we are aware, this is the first instance in the literature where an explanation of the above mentioned phenomenona can be given only assuming the Generalized Riemann Hypothesis; as mentioned above, previous work on the subject, such as the Rubinstein--Sarnak approach, also assume deep conjectures about the linear independence of zeros of the relevant $L$-functions. Finally, we refer the reader to \cite{AokiKoyama2023} for more examples of this flavour such as Chebyshev's bias in the splitting of prime ideals in Galois extension of number fields; we also note that there is related work by Okumura \cite{Okumura2023}, studying Chebyshev's bias for Fermat curves of primes degree. Using Theorem \ref{satakebias}, analogs for all these results can be proven outside a finite set of logarithmic measure assuming only the Generalized Riemann Hypothesis.

\begin{comment}
\begin{remark}
Slightly sharper results (replacing the $O(1)$ by $M+o(1)$ where $M$ is a constant) have been obtained by Aoki-Koyama under Conjecture \ref{DRH}; our examples are the analogous results under the assumption of the Generalized Riemann Hypothesis (as opposed to assuming Conjecture \ref{DRH}). 
\end{remark}
\end{comment}

%\subsection*{Notational conventions} We write $f=O(g)$ or $f \ll g$ if there exists a positive constant $c$ such that $|f(z)| \leq c |g(z)|$ for all $z$ in a specified range. The constant is allowed to depend on the automorphic representation.
%We write $f \asymp g$ if $f \ll g$ and $g \ll f$.  %We assume that $\pi$ is not the trivial representation and that $L(s, \pi)$ is entire. 

\subsection*{Acknowledgements}
I would like to thank Shin-ya Koyama for asking the questions that led to this paper and for helpful suggestions. I would also like to thank Adam Harper and Nuno Arala Santos for helpful discussions. 

\newpage
\section{Preliminary background} \label{Section2}

In this section, we review the properties of automorphic $L$-functions that we will need in the rest of the paper. 
Let $\pi=\bigotimes'\pi_v$ be an irreducible cuspidal automorphic representation of $\GL_n(\mathbb A_\mathbb Q)$.  Outside a finite set of places, for each finite place $p$, $\pi_p$ is unramified and we can associate to $\pi_p$ a semisimple conjugacy class $\{A_\pi(p)\}$ in $\GL_n(\mathbb C)$. Such a conjugacy class is parameterized by its eigenvalues $\alpha_{1, p}, \ldots, \alpha_{n, p}$. The local Euler factors $L_p(s, \pi_p)$ are given by 
$$
L_p(s, \pi_p)=\det(1-A_\pi(p) p^{-s})^{-1}= \prod_{j=1}^n (1-\alpha_{j, p}p^{-s})^{-1}. 
$$
At the ramified finite primes, the local factors are best described by the Langlands parameters of $\pi_p$ (see for instance the Appendix in \cite{SarnakRudnick1996}). They are of the form $L_p(s, \pi_p) = P_p(p^{-s})^{-1}$, where $P_p(x)$ is a polynomial of degree at most $n$, and $P_p(0)$ = 1. We will in this case too write the local factors in the form above, with the convention that we now allow some of the $\alpha$'s to be zero. The global $L$-function attached to $\pi$ is given by
\begin{align*}
L(s, \pi)= \prod_p L_p(s, \pi_p) &=\prod_p  \prod_{j=1}^n (1-\alpha_{j, p} p^{-s})^{-1}, 
\end{align*}

By the works of Godement--Jacquet \cite{GodementJacquet1972} and Jacquet--Shalika \cite{JacquetShalika1981}, $L(s, \pi)$ defines a holomorphic function for  $s \in \mathbb C$ with $\Re(s)>1$ and admits a meromorphic continuation to the entire complex plane.  For $\Re(s)>1$  we have that 
\begin{equation} \label{h}
-\frac{L'(s, \pi)}{L(s, \pi)} = \sum_{n=1}^{\infty} \frac{\Lambda(n) a_\pi(n)}{n^s}, 
\end{equation}
where $\Lambda(n)$ is the von-Mangoldt function and 
$$
a_{\pi}(p^k)=\alpha_{1, p}^k+ \cdots +\alpha_{n, p}^k. 
$$

 Sarnak and Rudnick \cite{SarnakRudnick1996} have shown that if $\pi_p$ is unramified, then $|\alpha_{j, p}|< p^{\frac{1}{2}-\frac{1}{n^2+1}} $. In this paper, we will often assume the following stronger bound.

\begin{conjecture}[The Ramanujan-Petersson Conjecture]
For any $p$ such that $\pi_p$ is uramified, $|\alpha_{j, p}|= 1$ for all $j \in \{1, \ldots, n\}$. 
\end{conjecture}

The Ramanujan-Petersson Conjecture is known in certain cases (for instance when $n=1$ and by the works \cite{Deligne1974} and \cite{DeligneSerre1974}, for all modular $L$-functions of degree two). The completed $L$-function of $L(s, \pi)$ is given by 
$
\Lambda(s, \pi) = Q(\pi)^{s/2} L_{\infty}(s, \pi) L(s, \pi), 
$
where  $Q(\pi)$ is the conductor of the representation and the archimedean factor is given by 
$
L_\infty(s, \pi) =\prod_{j=1}^n \Gamma_\mathbb R(s+\mu_\pi(j)), 
$
where the $\mu_\pi(j)$'s are certain constants and $\Gamma_\mathbb R(s)=\pi^{-s/2} \Gamma(s/2)$. The completed $L$-function satisfies a functional equation
$$
\Lambda(s, \pi) = \epsilon(\pi) \Lambda(1-s, \pi^\vee), 
$$
where $\epsilon(\pi) \in \mathbb C$ is a complex number of absolute one and $\pi^\vee$ is the contragradient representation of $\pi$.  The set of trivial zeros of $L(s, \pi)$ is $\{-2k-\mu_{\pi}(j): k \in \mathbb N \text{ and } j \in \{1, \ldots, n\}\}$. The Riemann Hypothesis for $L(s, \pi)$ is the statement that all the non-trivial zeros of $L(s, \pi)$ are on the line $\textrm{Re}(s)=\frac{1}{2}$.

\section{Proof of Theorem \ref{Thm1} }

Throughought this section, we let $\pi$ be an irreducible cuspidal automorphic representation of $\textrm{GL}_n(\mathbb A_\mathbb Q)$ such that the associated $L$-function $L(s, \pi)$ is entire. We begin by establishing a suitable version of an explicit formula for $L(s, \pi)$.  

\begin{proposition}  \label{explicit}
Let $s \in \mathbb C \setminus{ \{\frac{1}{2} \} } $ be a complex number such that $L(s, \pi) \neq 0$. We have that 
$$
\sum_{n \leq x} \frac{\Lambda(n) a_\pi(n) }{n^s} = -m \cdot \frac{x^{ \frac {1}{2} } -s}{\frac{1}{2}-s}-\frac{L'(s, \pi)}{L(s, \pi)}-\sum_{\rho \neq \frac{1}{2} } \frac{x^{\rho-s}}{\rho-s}+ \sum_{j=1}^n  \sum_{k=0}^{\infty} \frac{x^{-2k-\mu_\pi(j)-s} }{2k+\mu_\pi(j)+s}, 
$$
where $m=\ord_{s=\frac{1}{2}} L(s, \pi)$, the sum over $\rho$ is taken over all non-trivial zeros of $L(s, \pi)$ (excluding $\rho=\frac{1}{2}$) counting multiplicity and is interpreted as $\lim_{T \to \infty} \sum_{ |\gamma| \leq T}$, and where the last term of the sum on the left hand side is weighted by half if $x$ is an integer. 
\end{proposition}

\begin{proof}
%This is a standard argument to prove explicit formulas in analytic number theory. 
By using Perron's formula and Equation \eqref{h}, we have that 
\begin{equation*}
\sum_{n \leq x} \frac{\Lambda(n) a_\pi(n)}{n^s} = \frac{1}{2 \pi i} \int_{c- i \infty}^{c+ i \infty} \left( \sum_{n=1}^{\infty} \frac{\Lambda(n) a_\pi(n)}{n^{s+z}} \right) \frac{x^z}{z} dz=  \frac{1}{2 \pi i} \int_{c- i \infty}^{c+ i \infty}  - \frac{L'(s+z, \pi)}{L(s+z, \pi)} \frac{x^z}{z} dz
\end{equation*}
for $c \in \mathbb R$ sufficiently large.  By shifting the contour to the left and applying Cauchy's residue theorem, the above integral equals the sum of residues of the integrand in the region $\Re(z)\leq c$.   
\begin{itemize}
\item When $\displaystyle{s+z=\frac{1}{2}, \res_{z=\frac{1}{2}-s} \left( \frac{L'(s+z, \pi)}{L(s+z, \pi)} \right)=m}$,  so   $\displaystyle{\res_{z=\frac{1}{2}-s} \left( - \frac{L'(s+z, \pi)}{L(s+z, \pi)} \frac{x^z}{z} \right)=-m \cdot \frac{x^{ \frac {1}{2} } -s}{\frac{1}{2}-s}}$. 
\item When $z=0$, $\displaystyle{\res_{z=0} \left( - \frac{L'(s+z, \pi)}{L(s+z, \pi)} \frac{x^z}{z} \right)=-\frac{L'(s, \pi)}{L(s, \pi)}}$. 
\item If $\rho \neq \frac{1}{2}$ is a non-trivial zero of $L(s, \pi)$, $\displaystyle{\res_{z=\rho-s} \left( - \frac{L'(s+z, \pi)}{L(s+z, \pi)} \frac{x^z}{z} \right)=- r_{\rho} \cdot \frac{x^{\rho-s}}{\rho-s}}$, where $r_\rho$ is the order of the zero $\rho$. Thus, we obtain a total contribution of $\displaystyle{-\sum \limits _{\rho \neq \frac{1}{2}}\frac{x^{\rho-s}}{\rho-s}}$ counting multiplicity for all the non-trivial zeros of $L(s, \pi)$. 
\vspace{2mm}
\\
\item Finally, the trivial zeros of $L(s, \pi)$ are at $\{-2k-\mu_{\pi}(j): k \in \mathbb N \text{ and } j \in \{1, \ldots, n\}\}$. Thus,  since  $\displaystyle{\res_{z=-2k-\mu_\pi(j)-s} \left( - \frac{L'(s+z, \pi)}{L(s+z, \pi)} \frac{x^z}{z} \right)=\frac{x^{-2k-\mu_\pi(j)-s} }{2k+\mu_\pi(j)+s}}$ for all $k \geq 0$ and $j \in \{1, 2, \ldots, n \}$, we get a total contribution of $\displaystyle{ \sum_{j=1}^n  \sum_{k=0}^{\infty} \frac{x^{-2k-\mu_j(\pi)-s}}{2k+\mu_j(\pi)+s}}$. 
\end{itemize}
\end{proof}

We now note that 
\begin{align*}  \label{deuler}
\frac{d}{ds}  \left( \log \prod_{p \leq x} \prod_{j=1}^n (1-\alpha_{j, p}p^{-s})^{-1} \right) = -\sum_{j=1}^n \sum_{p \leq x}  \frac {\log p \cdot \alpha_{j,p}}{p^s-\alpha_{j,p}}  \\
\end{align*} 

and for each $j \in \{1, \ldots, n\}$, we write 

\begin{equation*}
\sum_{p \leq x}  \frac {\log p \cdot \alpha_{j,p}}{p^s-\alpha_{j,p}} = \sum_{\substack{p \leq x }}  \frac{\log p  \cdot \alpha_{j,p}} {p^s}+\sum_{\substack{p \leq x}}  \frac{\log p \cdot \alpha_{j,p}^2}{p^{2s}}+ \sum_{k \geq 3} \sum_{\substack{p \leq x }} \frac{\log p \cdot   \alpha_{j,p} ^{k} }{p^{ks}}. 
\end{equation*}

It thus follows that 
\begin{align*}
 \frac{d}{ds}  & \left( \log \prod_{p \leq x}  \prod_{j=1}^n (1-\alpha_{j, p}p^{-s})^{-1} \right) = -\sum_{j=1}^n \sum_{p \leq x}  \frac {\log p \cdot \alpha_{j,p}}{p^s-\alpha_{j,p}} = -\sum_{n \leq x} \frac{\Lambda(n) a_\pi(n) }{n^s} \\ &- \sum_{\substack{\sqrt x < p \leq x}} \frac{ \log p \cdot (\alpha_{1,p}^2+ \cdots +\alpha_{n,p}^2) }{p^{2s}} -\sum_{k \geq 3} 
 {\sum_{\substack{ x^{1/k}<p \leq x }}}\frac{ \log p \cdot  (\alpha_{1,p}^k+ \cdots +\alpha_{n,p}^k) }{p^{ks}}.
\end{align*}

\begin{theorem} \label{mainthm} 
Assume the Ramanujan--Petersson Conjecture and the Riemann Hypothesis for $L(s, \pi)$. Then for a complex number $s \in \mathbb C$ with $\frac{1}{2}< \Re(s)< 1$,  we have that 
\begin{equation*}
\prod_{p \leq x}  \prod_{j=1}^n (1-\alpha_{j, p} p^{-s})^{-1} = L(s, \pi)  \exp \left (-m \cdot \Li(x^{\frac{1}{2}-s}) - R_s(x) 
+ U_s(x)+ O\left( \frac{\log x}{x^{1/6}} \right)  \right), 
\end{equation*}
where we have set 
\begin{itemize}
\item $m=\ord \limits _{s=\frac{1}{2} } L(s, \pi)$,  \item $\Li(x)$ is the principal value of $\displaystyle{\int \limits _0^x \frac{dt}{ \log t}}$
\item $
\displaystyle{R_s(x) = \frac{1}{\log x} \sum \limits_{ \rho \neq \frac{1}{2} } \frac{x^{\rho-s}}{ \rho-s}+ \frac{1}{ \log x} \sum \limits_{\rho \neq \frac{1}{2} } \int_{s}^{\infty} \frac{x^{\rho-z}}{(\rho-z)^2}dz}$ 
\item 
$\displaystyle{U_s(x)= \sum \limits _{\substack{\sqrt x < p \leq x}} \frac{(\alpha_{1,p}^2+ \cdots +\alpha_{n,p}^2) }{2 p^{2s}}}$. 
\end{itemize}
Here, the sums in the term $R_s(x)$ are taken over all non-trivial zeros $\rho=\frac{1}{2} +i \gamma$ of $L(s, \pi)$ (excluding $\rho=\frac{1}{2}$) counted with multiplicity and are interpreted as $\lim _{T \to \infty} \sum  _{ |\gamma| \leq T}$,  and the integral is taken along the horizontal line starting at $s$.  
\end{theorem}

\begin{proof}
 For all $s \in \mathbb C$ with $\frac{1}{2}< \Re(s)< 1$,  we have using Proposition \ref{explicit} and the previous equation that
\begin{align*} 
\frac{d}{ds} &  \left( \log \prod_{p \leq x} \prod_{j=1}^n (1-\alpha_{j, p}p^{-s})^{-1} \right) =  m \cdot \frac{x^{ \frac {1}{2} } -s}{\frac{1}{2}-s}+\frac{L'(s, \pi)}{L(s, \pi)}+\sum_{\rho \neq \frac{1}{2}} \frac{x^{\rho-s}}{\rho-s}- \sum_{j=1}^n  \sum_{k=0}^{\infty} \frac{x^{-2k-\mu_\pi(j)-s} }{2k+\mu_\pi(j)+s} \\
&-\sum_{\substack{\sqrt x < p \leq x}} \frac{ \log p \cdot (\alpha_{1,p}^2+ \cdots +\alpha_{n,p}^2) }{p^{2s}}-\sum_{k \geq 3} 
 {\sum_{\substack{ x^{1/k}<p \leq x }}}\frac{ \log p \cdot (\alpha_{1,p}^k+ \cdots +\alpha_{n,p}^k) \cdot  }{p^{ks}} \\. 
\end{align*}
We now fix $s_0 \in \mathbb C$ with $\frac{1}{2} < \Re(s)< 1$. Integrating this equation along the horizontal straight line from $s_0$ to $\infty$, it follows that

\begin{align*} \label{logeulerproduct}
 \log \prod_{p \leq x} \prod_{j=1}^n (1-\alpha_{j, p}p^{-s_0})^{-1} &=  -m \int_{s_0}^{\infty} \frac{x^{ \frac {1}{2} } -s}{\frac{1}{2}-s}-\int_{s_0}^{\infty} \frac{L'(s, \pi)}{L(s, \pi)}-\sum_{\rho \neq \frac{1}{2} } \int_{s_0}^{\infty}  \frac{x^{\rho-s}}{\rho-s} \nonumber \\ &+ \int_{s_0}^{\infty} \sum_{j=1}^n  \sum_{k=0}^{\infty} \frac{x^{-2k-\mu_\pi(j)-s} }{2k+\mu_\pi(j)+s} 
+\sum_{\substack{\sqrt x < p \leq x}} \int_{s_0}^{\infty}  \frac{ \log p \cdot (\alpha_{1,p}^2+ \cdots +\alpha_{n,p}^2) }{p^{2s}}+ \nonumber \\ & +\sum_{k \geq 3} 
 {\sum_{\substack{ x^{1/k}<p \leq x }}} \int_{s_0}^{\infty}  \frac{ \log p \cdot  (\alpha_{1,p}^k+ \cdots +\alpha_{n,p}^k) \cdot  }{p^{ks}} \nonumber \\
\end{align*}

These integrals can now be analysed in an analogous way as in \cite[Theorem 2.3]{Sheth2023} to obtain the desired formula for the partial Euler product in the critical strip.  
\end{proof}

\begin{remark}
The implied constant in the big O term above (and in all big O terms appearing in the paper) depends only the automorphic representation $\pi$ and is independent of $s$. 
\end{remark}

Let $\psi(x, \pi)=\sum_{n \leq x} \Lambda(n) a_n(\pi)$.  It is known that, under the Riemann Hypothesis for $L(s, \pi)$, 
$$
\psi(x, \pi)= O( \sqrt x (\log x)^2). 
$$
Using a method introduced by Gallagher in \cite{Gallagher1980}, it is possible to improve the error term outside a set of finite logarithmic measure.

\begin{theorem} \label{refinement}
Let $\pi$ be an irreducible cuspidal representation of $\GL_n(\mathbb A_\mathbb Q)$ such that $L(s, \pi)$ is entire. Assume the Riemann Hypothesis for $L(s, \pi)$.  Then there exists a set $S$ of finite logarithmic measure such that for all $x \not \in S$, 
$$
\psi(x, \pi) \ll  \sqrt x (\log \log x)^2. 
$$
\end{theorem}

\begin{proof}
The  analog of this result for the prime counting function $\psi(x)$ was first proven by Gallagher \cite{Gallagher1980} and was later generalized in the setting of the theorem by Qu (\cite[Theorem 1.1]{Qu2007}); similar arguments have also been given in \cite[Theorem 2]{Koyama2016} and \cite[Theorem 3.4]{Sheth2023}. 
\end{proof}

As mentioned in the introduction,  we let $R(\pi) =\ord_{s=1} L_2(s, \pi)$ where $L_2(s, \pi)$ is the second-moment $L$-function associated to $L(s, \pi)$ and is defined by 

$$
L_2(s, \pi) = \prod_p \prod_{j=1}^n (1-\alpha_{j, p}^2p^{ -s} )^{-1}. 
$$
As explained in \cite[Example 1]{Devin2020}, there exists an open subset $U \supseteq \{s \in \mathbb C: \Re(s) \geq 1\}$ such $L_2(s, \pi)$ can be continued to a meromorphic function on $U$; thus, $R(\pi)$ is well-defined. We recall that we set $\nu(\pi) = m(\mathrm{sym}^{2} \pi)-m(\wedge^{2} \pi) \in \Z$,  where $m(\rho)$ denotes the multiplicity of the trivial representation $\textbf{1}$ in $\rho$ and that $\nu(\pi)=-R(\pi)$. 

\begin{lemma} \label{mertens}
There is a constant $M$ such that 
$$
\sum_{p \leq x} \frac{\alpha_{1, p}^2+\cdots+\alpha_{n, p}^2}{p}=-R(\pi) \log \log x +M+ o(1). 
$$

\end{lemma}

\begin{proof}
This follows from the work of Conrad, see \cite[page 275]{Conrad2005}. 
\end{proof}

\begin{remark}
Lemma \ref{mertens} can be regarded as a generalization of Merten's estimate 
$$
\sum_{p \leq x} \frac{1}{p} = \log \log x+ M +o(1). 
$$
Indeed, if all the $\alpha_{i, p}=1$ so that $L_2(s, \pi)$ is the Riemann zeta function, we recover Merten's estimate. 
\end{remark}

Lemma \ref{mertens} is essentially the main reason why $\sqrt 2$ shows up in the Euler product asymptotics. 

\begin{corollary} \label{sqrt2}
Let $U_s(x)$ be as in Theorem \ref{mainthm}. 
We have that 
$$
\lim _{x \to \infty} U_{\frac{1}{2}}(x) = -R(\pi) \cdot  \log \sqrt 2= \nu(\pi) \cdot \log \sqrt 2
$$
\end{corollary}

\begin{proof}
We have that 
\begin{align*}
U_{ \frac{1}{2}}(x) &= \sum_{p \leq x} \frac{\alpha_{1, p}^2+\cdots+\alpha_{n, p}^2}{2p}- \sum_{p \leq \sqrt x} \frac{\alpha_{1, p}^2+\cdots+\alpha_{n, p}^2}{2p} \\
& =\left (- \frac{R(\pi)}{2} \log \log x+ M +o(1) \right)- \left (-\frac{R(\pi)}{2} \log \log \sqrt x+ M +o(1) \right) \\
&= -R(\pi) \cdot \log \sqrt 2 +o(1). 
\end{align*}
\end{proof}

We can now prove our first main theorem.

\begin{theorem} [Theorem \ref{Thm1}] \label{boundedness}
Let $\pi$ be an irreducible cuspidal automorphic representation of $\GL_n(\mathbb A_\mathbb Q)$ such that $L(s, \pi)$ is entire and let $m = \ord_{s = 1/2} L(s, \pi)$. 
Assume the Ramanujan--Petersson conjecture and the Riemann Hypothesis for $L(s, \pi)$. Then there exists a subset $S \subseteq \mathbb R_{\geq 2}$ of finite logarithmic measure such that for all $x \not \in S$, 
$$
(\log x)^m \cdot  \prod_{p \leq x}  \prod_{j=1}^n (1-\alpha_{j, p}p^{-\frac{1}{2} })^{-1}   \sim  \frac{\sqrt{2}^{ \nu (\pi) }}{e^{m \gamma} m!} \cdot L^{(m)} \left(\pi, \frac{1}{2}\right).
$$

\end{theorem}

\begin{proof}
We follow the method of \cite[Theorem 4.2]{Sheth2023}.  When $s=\frac{1}{2}+\frac{1}{x}$, the left-hand side of Theorem \ref{mainthm} equals 
\begin{align*}
\prod_{p \leq x} \prod_{j=1}^n (1-\alpha_{j, p} p^{-s})^{-1} &= \prod_{p \leq x} \prod_{j=1}^n (1-\alpha_{j, p} p^{-\frac{1}{2}-\frac{1}{x}})^{-1} \\
&= \prod_{p \leq x} \prod_{j=1}^n \left (1-\alpha_{j, p} p^{-\frac{1}{2} } \left (1+ O \left ( \frac{\log p}{x} \right ) \right)  \right)^{-1}  \\
&= \prod_{p \leq x} \prod_{j=1}^n (1-\alpha_{j, p} p^{-\frac{1}{2}})^{-1} \cdot \prod_ p \prod_{j=1}^n \left(1 + O \left ( \frac{\alpha_{j, p} \cdot \log p }{ \sqrt p  x}  \right) \right), 
\end{align*}
where we used the fact that 
$$
(1+a+O(f(x)))^{-1}= (1+a)^{-1} (1+O(f(x))) \textrm{ if } f(x)=o(1) \textrm{ and } a  \textrm{ is sufficiently small.  }
$$
Now 
$$
\prod_ {p \leq x} \prod_{j=1}^n \left(1 + O \left ( \frac{\alpha_{j, p} \cdot \log p }{ \sqrt p  x}  \right) \right) = 1+O \left (  \sum_ {p \leq x} \frac{|\alpha_{j, p}| \cdot \log p}{\sqrt p x} \right )  =1+o(1), 
$$
where we used the Ramanujan--Petersson Conjecture and the estimate $\sum_{p \leq x} \frac{1}{\sqrt p} \ll \frac{\sqrt x}{\log x} $. 
Thus, when $s=\frac{1}{2}+\frac{1}{x}$, 
\begin{equation} \label{lhs}
\prod_{p \leq x} \prod_{j=1}^n (1-\alpha_{j, p} p^{-s})^{-1} \sim  \prod_{p \leq x}  \prod_{j=1}^n (1-\alpha_{j, p}p^{-\frac{1}{2} })^{-1}. 
\end{equation}

We now estimate the right-hand side of Theorem \ref{boundedness}.  Write 
$$
L(s, \pi)= a_m \left( s-\frac{1}{2} \right)^m+ a_{m+1}  \left( s-\frac{1}{2} \right) ^{m+1}+ \cdots 
$$
so that when $s=\frac{1}{2}+\frac{1}{x}$,  $L(s, \pi) \sim a_m \cdot \frac{1}{x^m}$ as $x \to \infty$. 

To estimate the contribution coming from the term $\Li(x^{\frac{1}{2}-s})$, we use the classical fact (see for instance \cite[pp.425]{Finch2003} or  \cite[Equation (2.2.5)]{Hardy1940}) that
$$
\Li(x) = \gamma + \log |\log x| +\sum_{n=1}^{\infty} \frac{ (\log x)^n}{n! \cdot n} \textrm{ for  all } x \in \mathbb  R_{>0} \setminus \{1\} .
$$
Applying this when $s=\frac{1}{2}+\frac{1}{x}$ yields that 
$$
\Li(x^{\frac{1}{2}-s}) = \gamma+ \log \left (\frac{\log x}{x} \right )+ o(1). 
$$

To estimate the contribution coming from the term $U_s(x)$ we note that $U_{\frac{1}{2} +\frac{1}{x}}(x) \to U_{\frac{1}{2}}(x)$ as $x \to \infty$ since, by a similar argument as above, we have that 
\begin{align*}
\sum_{{\substack{\sqrt x < p \leq x}}} \frac{(\alpha_{1,p}^2+ \cdots + \alpha_{n,p} ^2) }{2p^{2(\frac{1}{2}+\frac{1}{x}) }}&= \sum_{{\substack{\sqrt x < p \leq x }}} \frac{(\alpha_{1,p}^2+ \cdots + \alpha_{n,p} ^2) }{2p} \cdot \left(1+ O\left (\frac{\log p}{x} \right) \right) \\
&= \sum_{{\substack{\sqrt x < p \leq x}}} \frac{(\alpha_{1,p}^2+\cdots \alpha_{n, p}^2) }{2p}+ O\left (\frac{1}{x} \sum_{\sqrt x < p \leq x} \frac{ \log p}{p} \right)= U_{\frac{1}{2}}(x)+ o(1). 
\end{align*}
Using Corollary \ref{sqrt2}, we conclude that when $s=\frac{1}{2} +\frac{1}{x}$, $U_s(x) \to \nu(\pi) \cdot \log \sqrt 2$ as $x \to \infty$. 

To estimate the contribution coming from the term $$R_s(x)= \frac{1}{\log x}\sum \limits_{ \rho \neq \frac{1}{2} } \frac{x^{\rho-s}}{ \rho-s}+\frac{1}{\log x} \sum \limits_{\rho \neq \frac{1}{2}} \int_{s}^{\infty} \frac{x^{\rho-z}}{(\rho-z)^2}dz$$  when $s=\frac{1}{2}+\frac{1}{x}$,  we begin by noting that 
$$
\sum_{\rho \neq \frac{1}{2}} \frac{1}{\rho-s}-\sum_{\rho \neq \frac{1}{2} } \frac{1}{\rho} = \sum_{\rho \neq 1} \frac{s}{ (\rho-s) \rho}
$$
and that $$\sum_{\rho \neq \frac{1}{2} } \frac{s}{ (\rho-s) \rho} \ll  \sum_{\rho} \frac{1}{|\rho|^2} < \infty,
$$

where the fact that above sum converges follows, for instance, from the discussion in \cite[Remark 1]{Devin2020}.  Thus, 
\begin{equation*}
\sum_{\rho \neq \frac{1}{2}} \frac{1}{\rho-s} =\sum_{\rho} \frac{1}{\rho}+O(1)
\end{equation*}
and, since we are assuming the Riemann Hypothesis for $L(s, \pi)$, 
$$
\sum_{\rho \neq \frac{1}{2}} \frac{x^{\rho-s}}{\rho-s}=\frac{1}{x^s} \sum_{\rho} \frac{x^\rho}{\rho}+O(x^{\rho-s}) \ll \frac{1}{\sqrt x} \sum_{\rho} \frac{x^\rho}{\rho} +O(1). 
$$
Using the Riemann von-Mangoldt explicit formula for $L(s, \pi)$ (see  \cite[Theorem 3.1]{Qu2007})
$$
\psi(x, \pi) =- \sum_{\rho} \frac{x^{\rho}}{\rho}+O(x^{\frac{1}{2}-\frac{1}{n^2+1} } \log x), 
$$
it follows that 
$$
\sum_{\rho \neq \frac{1}{2}} \frac{x^{\rho-s}}{\rho-s} \ll \frac{1}{\sqrt x} \psi(x, \pi)+O(1). 
$$
By Theorem \ref{refinement}, we conclude that there exists a set $S$ of finite logarithmic measure such that for all $x \not \in S$, 
$$
\sum_{\rho \neq \frac{1}{2}} \frac{x^{\rho-s}}{\rho-s} \ll \frac{1}{\sqrt x} O( \sqrt x (\log \log x)^2)+O(1) 
$$
and so for all $x \not \in S$, 
$$
\frac{1}{\log x} \sum_{\rho \neq \frac{1}{2}} \frac{x^{\rho-s}}{\rho-s}  =o(1). 
$$
To estimate the second quantity in the definition of $R_s(x)$, we again use the fact that $\sum \limits_{\rho} \frac{1}{|\rho|^2}$ converges to conclude that 
\begin{align*}
\frac{1}{\log x} \sum_{\rho \neq \frac{1}{2}} \int_{s}^{\infty} \frac{x^{\rho-z}}{(\rho-z)^2}dz &= O \left(\frac{1}{\log x} \cdot \int_ {s}^{\infty}  x^{\frac{1}{2}-\Re(z) } d|z| \right) \\
&= O \left ( \frac{x^{\frac{1}{2} -\textrm{Re}(s)}}{\log^2 x}  \right) 
=o(1).  \\ 
\end{align*}

%The final term $ O \left( \frac{x^{1-2\Re(s)}-x^{ \frac{1-\Re(s)}{2} }}{(1-2\Re(s)) \log x} \right ) $ equals $O(1)$ when $s=\frac{1}{2}+\frac{1}{x}$, as can be seen by Taylor expanding the numerator. 
Thus, in summary, when $s=\frac{1}{2} +\frac{1}{x}$, we conclude that for all $x \not \in S$  we have that 
\begin{align*}
 & L(s, \pi)   \exp \left (-m \cdot \Li(x^{\frac{1}{2}-s}) -R_s(x)+U_s(x) + O \left (\frac{\log x}{x^{1/6}} \right) \right) 
  \\  &= L(s, \pi) \exp \left (-m \gamma -m \log \left (\frac{\log x}{x} \right )+\nu(\pi) \cdot \log \sqrt 2 + o(1) \right ) \\
& \sim  \frac{a_m}{x^m} \cdot \exp \left (-m \gamma -m\log \left (\frac{\log x}{x} \right)+ \nu (\pi) \cdot \log \sqrt 2  \right)=   \frac{1}{ (\log x)^m } \cdot  \frac{\sqrt{2}^{ \nu (\pi) }}{e^{m \gamma} m!} \cdot  L^{ (m) } \left( \frac{1}{2}, \pi \right)   \\
\end{align*}
Combining this with Equation \eqref{lhs} proves the theorem.  \qedhere

\end{proof}

\section{Applications to Chebyshev's Bias} \label{Section4}

In this section, we use Theorem \ref{boundedness} to obtain results towards Chebyshev's bias in the framework of \cite{AokiKoyama2023}.  We recall the following definition. 

\begin{definition}
Let $(c_p)_p \subseteq \R$ be a sequence over primes $p$ such that
\begin{equation*}
\lim_{x \to \infty} \frac{\#\{p \mid c_p > 0,  p \leq x \}}{\#\{p \mid c_p < 0,  p \leq x \}} = 1.
\end{equation*}
We say that $(c_p)_p$ has a \textit{Chebyshev bias towards being positive} if there exists a positive constant C
\begin{equation*}
\sum_{p \leq x} \frac{c_p}{\sqrt p} \sim  C \log \log x. 
\end{equation*}
On the other hand, we say that $c_p$ is \textit{unbiased} if
\begin{equation*}
\sum_{p \leq x} \frac{c_p}{\sqrt p } = O(1).
\end{equation*}
\end{definition}

%In this section, we let $\pi$ be an irreducible cuspidal automorphic representation with $L$-function
%$$
%L(s, \pi) = \prod_p \prod_{j=1}^n (1-\alpha_{j, p} p^{-s})^{-1}. 
%We assume that the second moment $L$-function $L_2(s, \pi)$ extends to a holomorphic function on the line $\textrm{Re}(s)=1$ and we let $R_0=\ord_{s=1} L_2(s, \pi)$. 

As mentioned in the introduction, the sum appearing in this definition is closely related to partial Euler products of $L$-functions at the centre of the critical strip. Indeed, using Theorem \ref{boundedness}, we obtain the following asymptotic.

\begin{theorem}[Chebyshev's bias for Satake parameters] \label{Satake}
Let $\pi$ be an irreducible cuspidal automorphic representation of $\GL_n(\mathbb A_\mathbb Q)$ such that $L(s, \pi)$ is entire and let $m =\ord_{s=1/2} L \left(s, \pi \right)$. Assume the Riemann Hypothesis and the Ramanujan-Petersson Conjecture for $L(s, \pi)$. Then there exists a constant $c_\pi$ such that 
$$
\sum_{p \leq x} \frac{\alpha_{1, p}+ \cdots +\alpha_{n, p}}{\sqrt p}= \left( \frac{R(\pi)}{2}-m \right) \log \log x+ c_\pi+ o(1)
$$
for all $x$ outside a set of finite logarithmic measure, where $R(\pi)=\ord \limits _{s=1} L_2(s, \pi)$. 
\end{theorem}

\begin{proof}
By Theorem \ref{boundedness}, we have that 
$$
(\log x)^m \cdot   \prod_{p \leq x}  \prod_{j=1}^n (1-\alpha_{j, p} p^{-\frac{1}{2}})^{-1} =  \frac{\sqrt{2}^{ \nu (\pi) }}{e^{m \gamma} m!} \cdot L^{(m)} \left(\frac{1}{2}, \pi \right)+o(1)
$$
for all $x$ outside a set $S$ of finite logarithmic measure. Taking logarithms yields that 
$$
 m \log \log x+  \sum_{p \leq x} \frac{\alpha_{1, p}+ \cdots +\alpha_{n, p}}{\sqrt p} + \sum_{p \leq x} \frac{\alpha_{1, p}^2+\cdots+\alpha_{n, p}^2}{2p}+ \sum_{p \leq x} \sum_{k \geq 3} \frac{\alpha_{1, p}^k+ \cdots +\alpha_{n, p}^k}{k p^{k/2 }} =  C+ o(1)
 $$
for all $x \not \in S$, where $C=\log \left( \frac{\sqrt{2}^{ \nu (\pi) }}{e^{m \gamma} m!} \cdot L^{(m)} \left(\frac{1}{2}, \pi \right) \right) $. 
Using Lemma \ref{mertens} and the fact that the last term on the left-hand side above converges, we get as desired that  
$$
\sum_{p \leq x} \frac{\alpha_{1, p}+ \cdots +\alpha_{n, p}}{\sqrt p}= \left( \frac{R(\pi)}{2}-m \right) \log \log x+ c_\pi+ o(1)
$$
for some constant $c_\pi$. 
\end{proof}

We now proceed to record a number of special cases of the above theorem. 

\subsection{The mod $4$ prime number race}

\begin{corollary} \label{primenumberrace}
Let $\chi_4$ denote the non-trivial Dirichlet character modulo 4.
Assume the Riemann Hypothesis for $L(\chi_4, s)$. 
Then there exists a constant $c$ such that 
$$
\sum_{p \leq x} \frac{\chi_4(p)}{\sqrt p}= -\frac{1}{2} \log \log x+ c+ o(1)
$$
for all $x$ outside a set of finite logarithmic measure.  In particular, 
$$
\pi_{\frac{1}{2}}(x; 4, 3)- \pi_{\frac{1}{2}}(x; 4, 1)=  \frac{1}{2} \log \log x+ c+ o(1)
$$
for all $x$ outside a set of finite logarithmic measure. 
\end{corollary}

\begin{proof}
It is well known that $L(\chi_{4}, \frac{1}{2} ) \neq 0$ and by Example \ref{Dirichlet}, $R(\chi_4)=-1$. The result thus follows from Theorem \ref{Satake}. \qedhere

\begin{comment}Thus, by Theorem \ref{boundedness}, there exist positive constants $C_1$ and $C_2$ such that 
$$
C_1 \leq \prod_{p \leq x} (1-\chi_{4}(p)p^{-\frac{1}{2}})^{-1} \leq C_2
$$
for all $x$ outside a set of finite logarithmic measure. 
Taking logarithms yields that 
\begin{equation} \label{primenumberrace}
\log C_1 \leq \sum_{p \leq x} \frac{\chi_4(p)}{\sqrt p}+ \sum_{p \leq x} \frac{\chi_4(p)^2}{2p}+ \sum_{p \leq x} \sum_{k \geq 3} \frac{\chi_4(p)^k}{k p^{k/2}} \leq \log C_2.  
\end{equation}
We note that 
\begin{itemize}

\item $\sum_{p \leq x} \frac{\chi_4(p)}{\sqrt p}=  - (\pi_{\frac{1}{2}}(x; 4, 3)-\pi_{\frac{1}{2}}(x; 4, 1))$

\item  $\sum_{p \leq x} \frac{\chi_4(p)^2}{2p}=\frac{1}{2} \log \log x+ O(1)$ by Merten's theorem. 

\item $
\sum_{k \geq 3} \frac{\chi_4(p)^k}{k p^{k/2}} \leq \sum_{n=1}^{\infty} \frac{1}{n^{3/2}}<\zeta(\frac{3}{2})<\infty$. 

\end{itemize}

Inserting these estimates into Equation \eqref{primenumberrace} proves the theorem. 
\end{comment}
\end{proof}

\subsection{Chebyshev's bias for Ramanujan's $\tau$ function}

We consider the $\Delta$ function defined by 
$$
\Delta(q)=q\prod_{n=1}^{\infty} (1-q^n)^{24}=\sum_{n=1}^{\infty} \tau(n) q^n
$$
where $q= e^{2 \pi i z}$ and $z \in \mathbb H= \{z \in \mathbb C: \Im(z)>0\}$. 
The $L$-function attached to $\Delta$ is defined to be 
$$
L(s, \Delta)=\sum_{n=1}^{\infty} \frac{\tau(n)}{n^s}=\prod_p (1-\tau(p)p^{-s
}+p^{11-2s})^{-1}  \hspace{5mm} \text{for } \Re(s)>\frac{13}{2} 
$$
and has an analytic continuation to the entire complex plane. 

\begin{comment}
By the work of Deligne, $|\tau(p)| \leq 2 p^{11/2}$, so we can write $\tau(p)=2p^{11/2} \cos (\theta_p)$ for a unique $\theta_p \in [0, \pi]$. An elementary calculation shows that we have 
$$
L \left(s+\frac{11}{2}, \Delta \right)=\prod_p \det(1-M(p)p^{-s})^{-1}, 
$$
where 
$$
M(p)= \begin{pmatrix}
    e^{i \theta_p} & 0 \\ 0 & e^{-i \theta_p}
    \end{pmatrix}. 
$$
\end{comment}

\begin{corollary} \label{tauasymptotic} 
Assume the Riemann Hypothesis for $L(s, \Delta)$.  Then there exists a constant $c$ such that 
$$
\sum_{p \leq x} \frac{ \tau(p) }{p^6} =  \frac{1}{2} \log \log x+ c+ o(1) 
$$
for all $x$ outside a set of finite logarithmic measure. 
\end{corollary}

\begin{proof}
By the work of Deligne \cite{Deligne1974} we have that $|\tau(p)| \leq 2 p^{11/2}$, so we can write $\tau(p)=2p^{11/2} \cos (\theta_p)$ for a unique $\theta_p \in [0, \pi]$. Let $\pi_\Delta$ denote the automorphic representation corresponding to $\pi$; then we have that 
$
\displaystyle{L(s, \pi_\Delta)= L \left(s+\frac{11}{2}, \Delta \right)}. 
$
An elementary calculation shows that we have 
$$
L \left(s+\frac{11}{2}, \Delta \right)=\prod_p (1-e^{i\theta_p} p^{-s} )^{-1}(1-e^{-i\theta_p} p^{-s})^{-1}. 
$$
It is known that $L \left(s, \pi_\Delta \right) \neq 0$ and $R(\pi_\Delta)=1$ (see \cite{KK22}). Thus, the asymptotic follows from Theorem \ref{Satake} since 
$$
\sum_{p \leq x} \frac{e^{i \theta_p}+e^{-i \theta_p}}{\sqrt p}= \sum_{ p \leq x} \frac{2 \cos (\theta_p) }{\sqrt p}= \sum_{p \leq x} \frac{ \tau(p) }{p^6}.  
$$

\begin{comment}
By Theorem \ref{boundedness}, there exist postive constants $C_1$ and $C_2$ such that 
$$
C_1 \leq  \prod_{p \leq x} \det(1-M(p) p^{-\frac{1}{2}})^{-1} \leq C_2
$$
for all $x$ outside a set of finite logarithmic measure.  Taking logarithms yields that 
\begin{equation} \label{tau}
\log C_1 \leq  \sum_{p \leq x} \frac{\Tr(M(p))}{\sqrt p}+ \sum_{p \leq x} \frac{\Tr(M(p)^2)}{2p}+ \sum_{p \leq x}\sum_{k \geq 3}  \frac{\Tr(M(p)^k)}{kp^{k/2}} \leq \log C_2. 
\end{equation}
We have that 
$$
 \sum_{p \leq x} \frac{\Tr(M(p))}{\sqrt p}=\sum_{p \leq x} \frac{2 \cos \theta_p}{\sqrt p}=\sum_{p \leq x} \frac{\tau(p)}{p^6}, 
$$
$$
\sum_{p \leq x} \frac{\Tr(M(p)^2)}{2p} \sim  -\frac{1}{2} \log \log x
$$
and 
$$
\sum_{p \leq x}\sum_{k \geq 3}  \frac{\Tr(M(p)^k)}{kp^{k/2}}= O(1). 
$$
Inserting these estimates into Equation \eqref{tau} proves the theorem. 
\end{comment}
\end{proof}

\subsection{Chebyshev's bias for Frobenius traces of elliptic curves}

Let $E/\mathbb Q$ be an elliptic curve of conductor $N_E$ with Frobenius trace $a_p$ for each prime $p$ and rank $\textrm{rk}(E)$.

\begin{corollary}  \label{frobtrace}
Assume the Riemann Hypothesis for $L(E, s)$ and the Birch and Swinnerton-Dyer conjecture. Then there exists a constant $c$ depending on $E$ such that 
$$
\sum_{p \leq x} \frac{a_p}{p} =   \left( \frac{1}{2}- \rk(E) \right) \log \log x + c+ o(1) 
$$
for all $x$ outside a set of finite logarithmic measure. 
\end{corollary}

\begin{proof}

\begin{comment}
By Theorem \ref{boundedness}, there exist postive constants $C_1$ and $C_2$ such that 
$$
C_1  \leq (\log x)^{\text{rk}(E)}  \prod_{p \leq x} \det(1-M(p) p^{-\frac{1}{2}})^{-1} \leq C_2
$$
for all $x$ outside a set of finite logarithmic measure.  Taking logarithms yields that 
\begin{equation} \label{tau}
\log C_1 \leq \textrm{rk}(E) \log \log x + \sum_{p \leq x} \frac{\Tr(M(p))}{\sqrt p}+ \sum_{p \leq x} \frac{\Tr(M(p)^2)}{2p}+ \sum_{p \leq x}\sum_{k \geq 3}  \frac{\Tr(M(p)^k)}{kp^{k/2}} \leq \log C_2. 
\end{equation}
We have that 
$$
 \sum_{p \leq x} \frac{\Tr(M(p))}{\sqrt p}=\sum_{p \leq x} \frac{2 \cos \theta_p}{\sqrt p}=\sum_{p \leq x} \frac{a_p}{p}= 
$$
$$
\sum_{p \leq x} \frac{\Tr(M(p)^2)}{2p} \sim  -\frac{1}{2} \log \log x
$$
and 
$$
\sum_{p \leq x}\sum_{k \geq 3}  \frac{\Tr(M(p)^k)}{kp^{k/2}}= O(1). 
$$
Inserting these estimates into Equation \eqref{tau} proves the theorem.
\end{comment}

By the Hasse-Weil bound, we can write $a_p=2 \sqrt p \cos (\theta_p)$ for a unique $\theta_p \in [0, \pi]$. As before,  we have that the normalized $L$-function is given by 
$$
L \left(E, s+\frac{1}{2} \right )= \prod_{p |N_E} (1-a_p p^{-s-\frac{1}{2}})^{-1} \cdot \prod_{p \nmid N_E} (1-e^{i\theta_p} p^{-s} )^{-1}(1-e^{-i\theta_p} p^{-s})^{-1}. 
$$
By \cite[Example 4.7]{Conrad2005}, the second moment $L$-function associated to $L \left(E, s+\frac{1}{2} \right )$  has a simple zero at $s=1$. Thus, noting that the contribution from the factors at the bad primes can be absorbed into the constant, the asymptotic follows from Theorem \ref{Satake} since 
$$
\sum_{{\substack{p \leq x \\ p \nmid  N_E}}} \frac{e^{i \theta_p}+e^{-i \theta_p}}{\sqrt p}= \sum_{{\substack{p \leq x \\ p \nmid  N_E}}} \frac{2 \cos (\theta_p) }{\sqrt p}= \sum_{{\substack{p \leq x \\ p \nmid  N_E}}} \frac{ a_p }{p}. 
$$

\end{proof}


\begin{thebibliography}{BCDT01}

\bibitem[AK23]{AokiKoyama2023}
M.~Aoki and S.~Koyama, \emph{{C}hebyshev's bias against splitting and principal
  primes in global fields},  J. Number Theory (2023), 232-262. 


\bibitem[BSD65]{BirchSwinnertonDyer1965}
B.~J. Birch and H.~P.~F. Swinnerton-Dyer, \emph{Notes on elliptic curves.
  {I\hspace{-.1mm}I}}, J. Reine Angew. Math. \textbf{218} (1965), 79--108.

\bibitem[BCDT01]{BreuilConradDiamondTaylor2001}
C. ~Breuil, B. ~Conrad, F. ~Diamond, and R. ~Taylor, \emph{On the modularity of elliptic curves over $\mathbb Q$:
wild 3-adic exercises}, J. Amer. Math. Soc., \textbf{14} (2001), 843–939.  



\bibitem[Con05]{Conrad2005}
K.~Conrad, \emph{Partial {E}uler products on the critical line}, Canad. J.
  Math. \textbf{57} (2005), no.~2, 267--297.

\bibitem[Del74]{Deligne1974}
P.~Deligne, \emph{La conjecture de {W}eil. {I}}, Publ. Math. Inst. Hautes
  {\'{E}}tudes Sci. \textbf{43} (1974), 273--307.  

\bibitem[DS74]{DeligneSerre1974} P.~Deligne and J.-P. Serre, \emph{Formes modulaires de poids 1}, Ann. Sci. École Norm. Sup.  \textbf{7} (1974), 507–530.   

\bibitem[Dev20]{Devin2020}
L.~Devin.
\emph{Chebyshev's bias for analytic {$L$}-functions}.
 Math. Proc. Cambridge Philos. Soc., \textbf{169}(2020), no.1., 103--140.
  

\bibitem[Fin03]{Finch2003}
S.R.~Finch, \emph{Mathematical constants}, Cambridge University Press, 2003. 


  
\bibitem[Gal80]{Gallagher1980}
P.~Gallagher,  
\emph{Some consequences of the Riemann hypothesis}, Acta Arith. \textbf{37} (1980),
339–343.

\bibitem[GJ72]{GodementJacquet1972} R. Godement and H.  Jacquet, \emph{Zeta Functions of Simple Algebras}, Lecture Notes in Mathematics, vol. 260, Springer-Verlag, Berlin-New York, 1972. 

\bibitem[Gol82]{Goldfeld1982}
D.~Goldfeld, \emph{Sur les produits eul{\'{e}}rinans attach{\'{e}}d aux courbes
  elliptiques}, C. R. Acad. Sci. Paris \textbf{294} (1982), 471--474.

\bibitem[Har40]{Hardy1940}
G.H.~Hardy, \emph{Ramanujan: Twelve lectures on subjects suggested by his life and work}, Cambridge University Press (1940). 

\bibitem[JS81]{JacquetShalika1981}
H.~Jacquet, J.~Shalika, \emph{On Euler products and the classification of automorphic representations I}, Amer. J. Math. \textbf{103} (1981), 499-558.

\bibitem[Kac95]{Kaczorowski1995}
J.~Kaczorowski,
  \emph{On the Distribution
  of Primes (mod {$4$})}, Analysis \textbf{15} (1995), no.~2, 159--171.

\bibitem[KK23]{KanekoKoyama2023} 
I.~Kaneko, S.~Koyama, 
\emph{A new aspect of Chebyshev’s bias for elliptic curves over function fields}, Proceedings of the American Mathematical Society, \textbf{51} (2023), 5059-5068.

\bibitem[KKK22]{KanekoKoyamaKurokawa2022}
I.~Kaneko, S.~Koyama, and N.~Kurokawa, 
\emph{Towards the {D}eep {R}iemann {H}ypothesis for {$\mathrm{GL}_{n}$}}.
preprint, 17 pages, 2022.
\newblock \url{https://arxiv.org/abs/2206.02612}

\bibitem[KKK14]{KimuraKoyamaKurokawa2014}
T.~Kimura, S.~Koyama, and N.~Kurokawa, \emph{{E}uler products beyond the
  boundary}, Lett. Math. Phys. \textbf{104} (2014), no.~1, 1--19.

\bibitem[KT62]{KnapowskiTuran1962}
S.~Knapowski and P.~Tur\'{a}n,
  \emph{Comparative Prime-Number
  Theory. {I}. {I}ntroduction}, Acta Mathematica. Academiae Scientiarum
  Hungaricae \textbf{13} (1962), 299--314. 

 \bibitem[Koy16]{Koyama2016}
S.~Koyama, 
\emph{Refinement of prime geodesic theorem}, Proc. Japan Acad. Ser A Math. Sci. 92 (2016), \textbf{7}, 77--81. 

\bibitem[KK22]{KK22} S.\ Koyama and N.\ Kurokawa, 
{\it Chebyshev's bias for Ramanujan's $\tau$-function via the deep Riemann hypothesis},
Proc.\ Japan Acad.\ Ser.\ A Math.\ Sci.\ {\bf 98} (2022), no.\ 6, 35--39.    

\bibitem[Maz08]{Maz08} B.\ Mazur, 
{\it Finding meaning in error terms}, 
Bull.\ Amer.\ Math.\ Soc.\ (N.S.) {\bf 45} (2008), no.\ 2, 185--228.

\bibitem[Mon70]{Mon70} H. ~Montgomery, \emph{The zeta function and prime numbers},  Proceedings of the Queen’s Number Theory
Conference, 1979, Queen’s Univ., Kingston (1980).



\bibitem[Qu07]{Qu2007}
Y.~Qu,  \emph{The prime number theorem for automorphic $L$-functions for $\GL_m$}, Journal of Number Theory \textbf{122} (2007), 84--99. 

\bibitem[Oku23]{Okumura2023}
Y.~Okumura, \emph{Chebyshev's bias for Fermat Curves of prime degree}, preprint, 14 pages, 2023. \newblock \url{https://arxiv.org/abs/2307.05958}

\bibitem[RS94]{RubinsteinSarnak1994}
M.~Rubinstein and P.~Sarnak, \emph{{C}hebyshev's bias}, Experiment. Math.
  \textbf{3} (1994), no.~3, 173--197.

  \bibitem[SR96]{SarnakRudnick1996}
 Z.~Rudnick, P.~Sarnak,   \emph{Zeros of princiapl $L$-functions and random matrix theory}, Duke Jour. of Math. \textbf{81} (1996), 269--322 (special volume in honor of J. Nash).

\bibitem[Sar07]{Sar07} P.\ Sarnak, 
Letter to: Barry Mazur on ``Chebyshev's bias'' for $\tau(p)$, 2007.

{\tt \url{https://publications.ias.edu/sites/default/files/MazurLtrMay08.PDF}}
%

\bibitem[She23]{Sheth2023}
A.~Sheth,
\newblock 
\emph{Euler Product Asymptotics for $L$-functions of Elliptic Curves},   
preprint, 17 pages, 2023.
\newblock \url{https://arxiv.org/abs/2312.05236}

\bibitem[Wil95]{Wiles1995}
A.~Wiles, \emph{Modular elliptic curves and {F}ermat's last theorem}, Ann. of
  Math. (2) \textbf{141} (1995), no.~3, 443--551.

\end{thebibliography}
\end{document}